\documentclass[11pt]{article}
\usepackage{amssymb,url,amsmath,graphicx, amsthm,mathdots,bm}
\usepackage{accents}
\usepackage{tikz}
\usetikzlibrary{positioning, calc, chains}
\usetikzlibrary{shapes.misc}
\usetikzlibrary{shapes.symbols}
\usetikzlibrary{shapes.geometric}
\usetikzlibrary{shapes.arrows}
\usetikzlibrary{fit}
\usetikzlibrary{shadows}

\usepackage{pgfplots}
\usepackage{pgfplotstable}

\usepackage{subcaption}

\newtheorem{theorem}{Theorem}[section]

\newtheorem{lemma}{Lemma}[theorem]
\newtheorem{remark}[theorem]{Remark}

\newtheorem{corollary}[theorem]{Corollary}

\newcommand{\citep}[1]{\cite{#1}}
\newcommand{\bfc}{\mathbf{c}}

\newcommand{\bfb}{\mathbf{b}}
\newcommand{\bfp}{\mathbf{p}}
\newcommand{\bfv}{\mathbf{v}}
\newcommand{\bfw}{\mathbf{w}}
\newcommand{\bfq}{\mathbf{q}}
\newcommand{\bfz}{\mathbf{z}}
\newcommand{\bfy}{\mathbf{y}}
\newcommand{\bfx}{\mathbf{x}}

\newcommand{\bfalpha}{\mathbf{\alpha}}
\newcommand{\bfbeta}{\mathbf{\beta}}
\newcommand{\bfpi}{\mathbf{\Pi}}
\newcommand{\bftheta}{\mathbf{\Theta}}

\usepackage{algorithmic}
\usepackage{algorithm}

\DeclareMathOperator{\Bez}{Bez}
\DeclareMathOperator{\diag}{diag}
\DeclareMathOperator{\Span}{Span}

\title{Structured inversion of the Bernstein--Vandermonde matrix}
\date{}
\author{Larry Allen\thanks{Department of Mathematics, Baylor
    University; One Bear Place \#97328; Waco, TX 76798-7328.
    Email: larry\_alllen@baylor.edu.}
    \and
    Robert C.~Kirby\thanks{Department of Mathematics, Baylor
    University; One Bear Place \#97328; Waco, TX 76798-7328.
    Email: robert\_kirby@baylor.edu.}  
	}

\begin{document}

\maketitle

\begin{abstract}
Bernstein polynomials, long a staple of approximation theory and computational geometry, have also increasingly become of interest in finite element methods. Many fundamental problems in interpolation and approximation give rise to interesting linear algebra questions. When attempting to find a polynomial approximation of boundary or initial data, one encounters the Bernstein--Vandermonde matrix, which is found to be highly ill-conditioned. In \cite{allenkirby2020mass}, we used the relationship between monomial B\'{e}zout matrices and the inverse of Hankel matrices to obtain a decomposition of the inverse of the Bernstein mass matrix in terms of Hankel, Toeplitz, and diagonal matrices. In this paper, we use properties of the Bernstein--B\'{e}zout matrix to factor the inverse of the Bernstein--Vandermonde matrix into a difference of products of Hankel, Toeplitz, and diagonal matrices. We also use the nonstandard matrix norm defined in \cite{allenkirby2020mass} to study the conditioning of the Bernstein--Vandermonde matrix, showing that the conditioning in this case is better than in the standard 2-norm. Additionally, we use properties of multivariate Bernstein polynomials to derive a block $LU$ decomposition of the Bernstein--Vandermonde matrix corresponding to equispaced nodes on the $d$-simplex.
\end{abstract}



\section{Introduction}

Given data $\{f_j\}_{j=0}^n$ and distinct nodes $\{x_j\}_{j=0}^n$, the interpolation problem consists of finding a polynomial $p$ of degree $n$ that satisfies
\begin{equation}
p(x_j)=f_j
\end{equation}
for each $0\leq j\leq n$. If a basis for the space of polynomials of degree at most $n$ is chosen, then the interpolation problem can be expressed as a system of linear equations, where the coefficient matrix is a Vandermonde-like matrix \cite{gautschi1999orthogonal}. Recently, Bernstein polynomials have been considered as a tool for high-order approximation of partial differential equations via the finite element method \cite{ainsworth2011bernstein,duffy1982quadrature}, and the interpolant is often used as a polynomial approximation of initial or boundary data. The corresponding Bernstein--Vandermonde matrix is found to be highly ill-conditioned \cite{davis1975condition}, but the structure of the matrix has led to fast algorithms that avoid some of the issues that arise from the ill-conditioning. For example, Marco and Mart\'{i}nez \cite{marco2007bidiagonal} used the fact that the Bernstein--Vandermonde matrix is strictly totally positive \cite{gasca1992positivity} to obtain a bidiagonal factorization of the inverse, and Ainsworth and Sanchez \cite{ainsworth2016newton} adapted the standard divided difference algorithm for the monomial Vandermonde matrix \cite{burden2012difference} to the Bernstein basis.

In \cite{allenkirby2020mass}, we used the relationship between monomial B\'{e}zout matrices and the inverse of Hankel matrices described by Heinig and Rost \cite{heinig1984algebraic} to obtain a decomposition of the inverse of the Bernstein mass matrix in terms of Hankel, Toeplitz, and diagonal matrices. In this paper, we generalize an argument made by Kaplan \cite{kaplan2006bezout} to obtain a decomposition of the inverse of Vandermonde-like matrices in terms of their transpose, a diagonal matrix, and the corresponding B\'{e}zout matrix. When applied to the Bernstein basis, this gives a decomposition of the inverse of the Bernstein--Vandermonde matrix in terms of Hankel, Toeplitz, and diagonal matrices, which in turn leads to a fast algorithm for solving the interpolation problem. Additionally, we use the nonstandard matrix norm defined in \cite{allenkirby2020mass} to give an explanation for the relatively good performance of the Bernstein--Vandermonde matrix despite its massive condition number.

Bernstein polynomials also extend naturally to give a basis for multivariate polynomials of total degree $n$. Properties of Bernstein polynomials lead to special recursive blockwise-structure for finite element matrices \cite{kirby2017fast,kirby2012fast}. In this paper, we use this structure to obtain a block $LU$ decomposition of the Bernstein--Vandermonde matrix associated to equispaced nodes on the $d$-simplex.

\section{Inverse Formulas}
\label{sec:inverse}
\subsection{General Formulas}
\label{ssec:general}

For integers $m,n\geq 0$, let $\{ b^n_j(x) \}_{j=0}^n$ be a basis for the space of univariate polynomials of degree at most $n$, and let $\bfx\in\mathbb{R}^{m+1}$ with $\bfx_i<\bfx_{i+1}$ for each $0\leq i<m$, where boldface is used to distinguish the vector of nodes $\bfx$ from the indeterminate $x$ (a similar convention will be used to distinguish a polynomial $v$ from its vector of coefficients $\bfv$). The Vandermonde matrix $V^n(\bfx)$ associated to $\{ b^n_j(x) \}_{j=0}^n$ and $\bfx$ is the $(m+1)\times(n+1)$ matrix given by 
\begin{equation}
\label{eq:vandermonde}
V^n_{ij}(\bfx)=b^n_j(\bfx_i).
\end{equation} 
We consider the case $m=n$ so that $V^n(\bfx)$ is invertible. 

If $v$ and $w$ are polynomials of degree at most $n+1$ expressed in the basis $\{b^{n+1}_j(x)\}_{j=0}^{n+1}$, then the B\'{e}zout matrix generated by $v$ and $w$, denoted $\Bez(v,w)$, is the $(n+1)\times (n+1)$ matrix whose entries satisfy
\begin{equation}
\label{eq:bezout}
\frac{v(s)w(t)-v(t)w(s)}{s-t} = \sum_{i,j=0}^n \Bez_{ij}(v,w)b^n_i(s)b^n_j(t).
\end{equation}
In \cite{kaplan2006bezout}, it was shown that if the monomial basis is used, then the inverse of a B\'{e}zout matrix is a Hankel matrix. As a consequence of the arguments involved, one obtains a formula for the inverse of the monomial Vandermonde matrix in terms of its transpose, a diagonal matrix, and a particular B\'{e}zout matrix. Since none of the arguments require anything specific to the monomial basis, we can generalize the arguments to any basis, which we summarize here.

For $t\in\mathbb{R}$, let $\bfb^n(t)$ denote the column vector $(b^n_0(t),\dots,b^n_n(t))^T$. As a consequence of (\ref{eq:bezout}), if $s\neq t$, then
\begin{align*}
(\bfb^n(s))^T\Bez(v,w)\bfb^n(t) 
&= \sum_{i,j=0}^n \Bez_{ij}(v,w)b^n_i(s)b^n_j(t) \\
&= \frac{v(s)w(t)-v(t)w(s)}{s-t}.
\end{align*}
This implies that
\begin{align*}
(\bfb^n(t))^T\Bez(v,w)\bfb^n(t)
&= \lim_{\varepsilon\rightarrow 0} \left[ (\bfb^n(t))^T\Bez(v,w)\bfb^n(t+\varepsilon) \right] \\
&= \lim_{\varepsilon\rightarrow 0} \frac{v(t+\varepsilon)w(t)-v(t)w(t+\varepsilon)}{\varepsilon} \\
&= v'(t)w(t)-v(t)w'(t).
\end{align*}
Therefore, if $v$ has simple zeros at $\bfx_i$ and $w(\bfx_i)\neq 0$ for each $0\leq i\leq n$, then
\begin{equation}
V^n(\bfx)\Bez(v,w)\left(V^n(\bfx)\right)^T = \diag(v'(\bfx_j)w(\bfx_j))_{j=0}^n,
\end{equation}
and so we have the following:

\begin{theorem}
\label{thm:bezout}
The inverse of $V^n(\bfx)$ is given by
\begin{equation}
\left(V^n(\bfx)\right)^{-1} = \Bez(v,w)\left(V^n(\bfx)\right)^T\diag\left(\frac{1}{v'(\bfx_j)w(\bfx_j)}\right)_{j=0}^n,
\end{equation}
where $v$ and $w$ are any polynomials of degree $n+1$ satisfying $v(\bfx_i)=0$, $v'(\bfx_i)\neq 0$, and $w(\bfx_i)\neq 0$ for each $0\leq i\leq n$.
\end{theorem}

\begin{corollary}
\label{cor:bezout}
The inverse of $V^n(\bfx)$ is given by
\begin{equation}
\left(V^n(\bfx)\right)^{-1} = \Bez(v,1)\left(V^n(\bfx)\right)^T\diag\left(\frac{1}{v'(\bfx_j)}\right)_{j=0}^n,
\end{equation}
where $v$ is any polynomial of degree $n+1$ that has simple zeros at $\bfx_i$ for each $0\leq i\leq n$.
\end{corollary}

\subsection{Bernstein--Vandermonde}
\label{ssec:BernsteinVandermonde}

We now focus on the case where $0\leq \bfx_i\leq 1$ for each $0\leq i\leq n$ and the basis consists of the polynomials $\{B^n_j(x)\}_{j=0}^n$, where
\begin{equation}
\label{eq:bernstein}
B^n_j(x) = \binom{n}{j}x^j(1-x)^{n-j}
\end{equation}
are the Bernstein polynomials of degree $n$. In order to apply Theorem~\ref{thm:bezout} or Corollary~\ref{cor:bezout}, we need a method for computing the entries of the Bernstein--B\'{e}zout matrix; in fact, they satisfy a recurrence relation involving the Bernstein coefficients of the polynomials.

\begin{theorem}
\label{thm:recursion}
If $v(t)=\sum_{i=0}^{n+1} \bfv_iB^{n+1}_i(t)$ and $w(t)=\sum_{i=0}^{n+1} \bfw_iB^{n+1}_i(t)$ are polynomials of degree at most $n+1$, then the entries $b_{ij}$ of the Bernstein--B\'{e}zout matrix $\Bez(v,w)$ generated by $v$ and $w$ satisfy
\begin{equation}
\label{eq:recursion}
b_{ij} = \frac{1}{(i+1)(n-j+1)}\left[j(n-i)b_{i+1,j-1}+(n+1)^2\left( \bfv_{i+1}\bfw_j-\bfv_j\bfw_{i+1}\right)\right].
\end{equation}
\end{theorem}

\begin{proof}
By (\ref{eq:bernstein}), we have that
\begin{equation}
tB^n_j(t) = \frac{j+1}{n+1}B^{n+1}_{j+1}(t)
\end{equation}
and
\begin{equation}
B^n_j(t) = \frac{n-j+1}{n+1}B^{n+1}_j(t)+\frac{j+1}{n+1}B^{n+1}_{j+1}(t).
\end{equation}
Therefore,
\begin{align*}
s\sum_{i,j=0}^n b_{ij}B^n_i(s)B^n_j(t) &= \sum_{i,j=0}^n \frac{i+1}{n+1}  b_{ij}B^{n+1}_{i+1}(s)B^n_j(t) \\
&= \sum_{i,j=0}^n \frac{(i+1)(n-j+1)}{(n+1)^2}b_{ij} B^{n+1}_{i+1}(s)B^{n+1}_j(t) \\
&\qquad +\sum_{i,j=0}^n \frac{(i+1)(j+1)}{(n+1)^2}b_{ij} B^{n+1}_{i+1}(s)B^{n+1}_{j+1}(t).
\end{align*}
Similarly, we have that
\begin{align*}
t\sum_{i,j=0}^n b_{ij}B^n_i(s)B^n_j(t) &= \sum_{i,j=0}^n \frac{j+1}{n+1}b_{ij}B^n_i(s)B^{n+1}_{j+1}(t) \\
&= \sum_{i,j=0}^n \frac{(n-i+1)(j+1)}{(n+1)^2} b_{ij} B^{n+1}_i(s)B^{n+1}_{j+1}(t) \\
&\qquad + \sum_{i,j=0}^n \frac{(i+1)(j+1)}{(n+1)^2} b_{ij} B^{n+1}_{i+1}(s)B^{n+1}_{j+1}(t).
\end{align*}
This implies that
\begin{align*}
(s-t)\sum_{i,j=0}^n b_{ij} B^n_i(s)B^n_j(t) &= \sum_{i,j=0}^n \frac{(i+1)(n-j+1)}{(n+1)^2}b_{ij} B^{n+1}_{i+1}(s)B^{n+1}_j(t) \\
&\quad -\sum_{i,j=0}^n \frac{(n-i+1)(j+1)}{(n+1)^2}b_{ij}B^{n+1}_i(s)B^{n+1}_{j+1}(t),
\end{align*}
and so
\begin{align*}
(s-t)\sum_{i,j=0}^n b_{ij}B^n_i(s)B^n_j(t) &= \sum_{i,j=0}^{n+1}\frac{i(n-j+1)}{(n+1)^2}b_{i-1,j}B^{n+1}_i(s)B^{n+1}_j(t) \\
&\quad -\sum_{i,j=0}^{n+1}\frac{j(n-i+1)}{(n+1)^2}b_{i,j-1}B^{n+1}_i(s)B^{n+1}_j(t).
\end{align*}
On the other hand,
\begin{equation}
v(s)w(t)-v(t)w(s) = \sum_{i,j=0}^{n+1} (\bfv_i\bfw_j-\bfv_j\bfw_i)B^{n+1}_i(s)B^{n+1}_j(t),
\end{equation}
and so the result follows by comparing coefficients in (\ref{eq:bezout}).
\end{proof}

We can form the first column and last row of $\Bez(v,w)$ by using (\ref{eq:recursion}) with $j=0$ and $i=n$. The rest of the columns can then be built by applying the recurrence relation. Therefore,

\begin{corollary}
\label{cor:bezoutcost}
$\Bez(v,w)$ can be constructed in $\mathcal{O}(n^2)$ operations.
\end{corollary}

By repeatedly applying (\ref{eq:recursion}), we obtain a closed form for the entries of the Bernstein--B\'{e}zout matrix.

\begin{corollary}
\label{cor:bernbezout}
If $v(t)=\sum_{i=0}^{n+1} \bfv_iB^{n+1}_i(t)$ and $w(t)=\sum_{i=0}^{n+1} \bfw_iB^{n+1}_i(t)$ are polynomials of degree at most $n+1$, then the entries $b_{ij}$ of the Bernstein--B\'{e}zout matrix $\Bez(v,w)$ generated by $v$ and $w$ are given by
\begin{equation}
\label{eq:bernbezout}
b_{ij} = \frac{1}{\binom{n}{i}\binom{n}{j}}\sum_{k=0}^{m_{ij}} \binom{n+1}{i+k+1}\binom{n+1}{j-k}\left( \bfv_{i+k+1}\bfw_{j-k}-\bfv_{j-k}\bfw_{i+k+1}\right),
\end{equation}
where $m_{ij} = \min\{j,n-i\}$.
\end{corollary}

Constructing $\Bez(v,w)$ by using (\ref{eq:bernbezout}) would require $\mathcal{O}(n^3)$ operations, and so Corollary~\ref{cor:bezoutcost} implies that (\ref{eq:bernbezout}) is not optimal for building $\Bez(v,w)$; however, since $\binom{n}{i}=0$ whenever $i>n$, the upper limit of the sum in (\ref{eq:bernbezout}) can be replaced with $k=n$, and so we recognize that $\Bez(v,w)$ is given by a difference of matrix products.

\begin{corollary}
\label{cor:factorization}
Let $v(t) = \sum_{i=0}^{n+1} \bfv_iB^{n+1}_i(t)$ and $w(t)=\sum_{i=0}^{n+1}\bfw_iB^{n+1}_i(t)$ be polynomials of degree at most $n+1$. Define the Toeplitz matrices $T^{v,n}$ and $T^{w,n}$ by
\begin{equation*}
T^{v,n}_{ij} = \binom{n+1}{j-i}\bfv_{j-i}\qquad \text{and}\qquad T^{w,n}_{ij} = \binom{n+1}{j-i}\bfw_{j-i},
\end{equation*}
define the Hankel matrices $H^{v,n}$ and $H^{w,n}$ by
\begin{equation*}
H^{v,n}_{ij} = \binom{n+1}{i+j+1}\bfv_{i+j+1}\qquad \text{and}\qquad H^{w,n}_{ij} = \binom{n+1}{i+j+1}\bfw_{i+j+1},
\end{equation*}
and let $\Delta^n=\diag\left(\binom{n}{j}\right)_{j=0}^n$. Then the Bernstein--B\'{e}zout matrix $\Bez(v,w)$ generated by $v$ and $w$ is given by
\begin{equation}
\label{eq:bezoutinverse}
\Bez(v,w) = \left(\Delta^n\right)^{-1}\left[ H^{v,n}T^{w,n}-H^{w,n}T^{v,n}\right]\left(\Delta^n\right)^{-1}.
\end{equation}
\end{corollary}

\begin{remark}
\label{rmk:connection}
In \cite{allenkirby2020mass}, it was shown that the inverse of the degree $n$ Bernstein mass matrix $M^n$ has the decomposition
\begin{equation}
\label{eq:massinverse}
\left(M^n\right)^{-1} = \left(\Delta^n\right)^{-1}\left[ \widetilde{T}^nH^n-T^n\widetilde{H}^n\right] \left(\Delta^n\right)^{-1},
\end{equation}
where $\widetilde{H}^n$ and $H^n$ are Hankel matrices, $\widetilde{T}^n$ and $T^n$ are Toeplitz matrices, and $\Delta^n = \diag\left(\binom{n}{j}\right)_{j=0}^n$. This decomposition was obtained from the monomial B\'{e}zout matrix corresponding to polynomials whose coefficients are contained in the last column of $\left(M^n\right)^{-1}$. It is interesting to note the similarities between (\ref{eq:bezoutinverse}) and (\ref{eq:massinverse}) despite the different matrices and bases being considered.
\end{remark}

Given a polynomial $v(t)=\sum_{k=0}^{n+1} \bfv_k B^{n+1}_k(t)$ of degree at most $n+1$, we can express $v$ in the monomial basis as $v(t)=\sum_{k=0}^{n+1} \widetilde{\bfv}_kt^k$. It was shown in \cite{farouki2000legendre} that
\begin{equation}
\label{eq:montobern}
\bfv_k = \frac{1}{\binom{n+1}{k}} \sum_{\ell=0}^k \binom{n-\ell+1}{k-\ell} \widetilde{\bfv}_{\ell}.
\end{equation}
This implies that
\begin{align*}
H^{v,n}_{ij} = \binom{n+1}{i+j+1}\bfv_{i+j+1} &= \sum_{k=0}^{i+j+1} \binom{n-k+1}{i+j-k+1}\widetilde{\bfv}_k \\
&= \sum_{k=0}^{i+j+1} \binom{n-k+1}{n-i-j}\widetilde{\bfv}_k.
\end{align*}
Similarly,
\begin{equation*}
T^{v,n}_{ij} = \sum_{k=0}^{j-i}\binom{n-k+1}{j-i-k}\widetilde{\bfv}_ k.
\end{equation*}
By Vieta's formula (see, for example, \cite{rotman2015algebra}),
\begin{equation}
\label{eq:vieta}
\prod_{i=0}^n (x-\bfx_i) = \sum_{k=0}^n (-1)^{n-k+1}\sigma_{n-k+1}(\bfx)x^k,
\end{equation}
where
\begin{equation}
\sigma_k(\bfx) = \begin{cases} 1, & \text{if}\ k=0; \\ \sum_{0\leq i_0<\dots<i_{k-1}\leq n} \bfx_{i_0}\cdots \bfx_{i_{k-1}}, & \text{otherwise}; \end{cases}
\end{equation}
is the $k^{\text{th}}$ elementary symmetric function in the $n+1$ variables $\bfx_0,\dots,\bfx_n$. In addition, since
\begin{equation}
1 = \sum_{i=0}^{n+1} B^{n+1}_i(x)
\end{equation}
and
\begin{equation}
\frac{d}{dx}\left[ \prod_{i=0}^n (x-\bfx_i)\right]_{x=x_j} = \prod_{i\in\{0,\dots,n\}\setminus\{j\}} (\bfx_j-\bfx_i),
\end{equation}
we can combine the previous discussion with Corollary~\ref{cor:bezout} and Corollary~\ref{cor:factorization} to obtain a decomposition of the inverse of the Bernstein--Vandermonde matrix.

\begin{theorem}
\label{thm:factoredinverse}
Define the Hankel matrices $H^n$ and $\widetilde{H}^n(\bfx)$ by
\begin{equation*}
H^n_{ij} = \binom{n+1}{i+j+1}\qquad \text{and}\qquad \widetilde{H}^n_{ij}(\bfx) = \sum_{k=0}^{i+j+1} (-1)^{n-k+1}\binom{n-k+1}{n-i-j}\sigma_{n-k+1}(\bfx),
\end{equation*}
define the Toeplitz matrices $T^n$ and $\widetilde{T}^n(\bfx)$ by
\begin{equation*}
T^n_{ij} = \binom{n+1}{j-i}\qquad \text{and}\qquad \widetilde{T}^n_{ij}(\bfx) = \sum_{k=0}^{j-i} (-1)^{n-k+1} \binom{n-k+1}{j-i-k}\sigma_{n-k+1}(\bfx),
\end{equation*}
and define the diagonal matrices $D^n(\bfx)$ and $\Delta^n$ by
\begin{equation*}
D^n(\bfx) = \diag\left( \prod_{i\in\{0,\dots,n\}\setminus\{j\}}(\bfx_j-\bfx_i)\right)_{j=0}^n\qquad \text{and}\qquad \Delta^n = \diag\left( \binom{n}{j}\right)_{j=0}^n.
\end{equation*}
In addition, let $\widetilde{V}^n(\bfx)$ be the scaled Bernstein--Vandermonde matrix
\begin{equation*}
\widetilde{V}^n_{ij}(\bfx) = \bfx_i^j(1-\bfx_i)^{n-j}.
\end{equation*}
Then
\begin{equation}
\left(V^n(\bfx)\right)^{-1} = \left(\Delta^n\right)^{-1}\left[ \widetilde{H}^n(\bfx)T^n-H^n\widetilde{T}^n(\bfx)\right]\left(\widetilde{V}^n(\bfx)\right)^T\left(D^n(\bfx)\right)^{-1}.
\end{equation}
\end{theorem}

\subsection{Equispaced Nodes}
\label{ssec:equispaced}

In Section~\ref{sec:dimensions}, we will derive a block $LU$ decomposition for the Bernstein--Vandermonde matrix associated to equispaced nodes on the $d$-simplex. This decomposition leads to a recursive, block-structured algorithm for the corresponding interpolation problem. The base for this algorithm is the one-dimensional Bernstein--Vandermonde matrix associated to equispaced nodes, and so we now briefly focus on the case where $\bfx_i=i/n$ for each $0\leq i\leq n$. For these nodes, we observe that (\ref{eq:vieta}) becomes
\begin{align*}
\prod_{i=0}^n (x-i/n) = \frac{1}{n^{n+1}}\prod_{i=0}^{n+1} (nx-i) &= \frac{1}{n^{n+1}}\sum_{k=0}^{n+1} s(n+1,k)(nx)^k \\
&= \sum_{k=0}^{n+1} \frac{s(n+1,k)}{n^{n-k+1}}x^k,
\end{align*}
where $s(n,k)$ are the (signed) Stirling numbers of the first kind \cite{stegun1965handbook}. In addition, we have that
\begin{equation}
\prod_{i\in\{0,\dots,n\}\setminus\{j\}}(j/n-i/n) = (-1)^{n-j} \frac{j!(n-j)!}{n^n}.
\end{equation}

\begin{corollary}
\label{cor:equispacedinverse}
Define the Hankel matrices $H^n$ and $\widetilde{H}^n$ by
\begin{equation*}
H^n_{ij} = \binom{n+1}{i+j+1}\qquad \text{and}\qquad \widetilde{H}^n_{ij} = \sum_{k=0}^{i+j+1}\binom{n-k+1}{n-i-j}\frac{s(n+1,k)}{n^{n-k+1}},
\end{equation*}
define the Toeplitz matrices $T^n$ and $\widetilde{T}^n$ by
\begin{equation*}
T^n_{ij} = \binom{n+1}{j-i}\qquad \text{and}\qquad \widetilde{T}^n_{ij} = \sum_{k=0}^{j-i} \binom{n-k+1}{j-i-k}\frac{s(n+1,k)}{n^{n-k+1}},
\end{equation*}
and define the diagonal matrices $D^n$ and $\Delta^n$ by
\begin{equation*}
D^n = \diag\left( (-1)^{n-j}\left[j!(n-j)!\right]\right)_{j=0}^n\qquad \text{and}\qquad \Delta^n = \diag\left( \binom{n}{j}\right)_{j=0}^n.
\end{equation*}
In addition, let $\widetilde{V}^n$ be the scaled Bernstein--Vandermonde matrix
\begin{equation*}
\widetilde{V}^n_{ij} = i^j(n-i)^{n-j}.
\end{equation*}
Then the inverse of $V^n=V^n(0,1/n,\dots,1)$ is given by
\begin{equation}
\left(V^n\right)^{-1} = \left(\Delta^n\right)^{-1}\left[ \widetilde{H}^nT^n-H^n\widetilde{T}^n\right]\left(\widetilde{V}^n\right)^T\left(D^n\right)^{-1}.
\end{equation}
\end{corollary}

\section{Applying the inverse}
\label{sec:apply}

Now, we describe several approaches to applying $\left(V^n(\bfx)\right)^{-1}$ to a vector.

\paragraph{$LU$ factorization} We can decompose $V^n(\bfx)$ as
\begin{equation}
V^n(\bfx) = L^n(\bfx)U^n(\bfx),
\end{equation}
where $L^n(\bfx)$ and $U^n(\bfx)$ are lower and upper triangular matrices, respectively. Widely available in libraries, computing $L^n(\bfx)$ and $U^n(\bfx)$ requires $\mathcal{O}(n^3)$ operations, and each of the subsequent triangular solves require $\mathcal{O}(n^2)$ operations to perform.

\paragraph{Exact inverse} In light of Corollary~\ref{cor:bezout}, we can directly form $\left(V^n(\bfx)\right)^{-1}$. By Theorem~\ref{thm:recursion}, we can form the inverse in $\mathcal{O}(n^2)$ operations. The inverse can then be applied to any vector in $\mathcal{O}(n^2)$ operations using the standard algorithm.

\paragraph{DFT-based application} By Theorem~\ref{thm:factoredinverse}, we can express the inverse of $V^n(\bfx)$ in terms of its (scaled) transpose and Hankel, Toeplitz, and diagonal matrices. The matrices $H^n(\bfx)$ and $T^n(\bfx)$ can be formed in $\mathcal{O}(n)$ operations, but the matrices $\widetilde{H}^n(\bfx)$ and $\widetilde{T}^n(\bfx)$ require $\mathcal{O}(n^2)$ operations to form, since each entry involves a sum. Even though the Hankel and Toeplitz matrices can be applied to a vector in $\mathcal{O}(n\log n)$ operations via circulant embedding and a couple of FFT/iFFT \cite{vogel2002fft}, the scaled transpose still requires $\mathcal{O}(n^2)$ operations to apply to a vector.

\paragraph{Newton algorithm} In \cite{ainsworth2016newton}, Ainsworth and Sanchez gave a fast algorithm for solving $V^n(\bfx)\bfc=\bfb$ via a recursion relation for the B\'{e}zier control points of the Newton form of the Lagrangian interpolant. The numerical results given suggest that the algorithm is stable even for high polynomial degree, and so we include it for comparison.

\section{Conditioning and accuracy}
\label{sec:cond}

In \cite{allenkirby2020mass}, we gave a partial explanation of the relatively high accuracy, compared to its large 2-norm condition number, of working with the Bernstein mass matrix. We will use similar reasoning to investigate the performance of the Bernstein--Vandermonde matrix.

Recall that for an integer $n\geq 0$, the Bernstein mass matrix $M^n$ is the $(n+1)\times(n+1)$ matrix given by
\begin{equation}
\label{eq:mass}
M^n_{ij} = \int_0^1 B^n_i(x)B^n_j(x)dx.
\end{equation}
The mass matrix (for Bernstein or any other family of polynomials) plays an important role in connecting the $L^2$ topology on the finite-dimensional space to linear algebra.  To see this, we first define mappings connecting polynomials of degree $n$ to $\mathbb{R}^{n+1}$.  Given any $\bfc \in \mathbb{R}^{n+1}$, we let $\pi(\bfc)$ be the polynomial expressed in the Bernstein basis with coefficients contained in $\bfc$:
\begin{equation}
\pi(\bfc)(x) = \sum_{i=0}^n \bfc_i B^n_i(x).
\end{equation}
We let $\bfpi$ be the inverse of this mapping, sending any polynomial of degree at most $n$ to the vector of $n+1$ coefficients with respect to the Bernstein basis.

Now, let $p(x)$ and $q(x)$ be polynomials of degree $n$ with expansion coefficients $\bfpi(p) = \bfp$ and $\bfpi(q) = \bfq$.  Then the $L^2$ inner product of $p$ and $q$ is given by the $M^n$-weighted inner product of $\bfp$ and $\bfq$, for 
\begin{equation}
\label{eq:innprod}
\int_0^1 p(x) q(x) dx
=
\sum_{i,j=0}^n \bfp_i \bfq_j \int_0^1 B^n_i(x) B^n_j(x) dx
= \bfp^T M^n \bfq.
\end{equation}
Similarly, if
\begin{equation}
\| \bfp \|_{M^n} = \sqrt{ \bfp^T M^n \bfp }
\end{equation}
is the $M^n$-weighted vector norm, then we have for $p = \pi(\bfp)$,
\begin{equation}
\label{eq:topology}
\| p \|_{L^2} = \| \bfp \|_{M^n}.
\end{equation}

We can interpret the Bernstein--Vandermonde matrix $V^n(\bfx)$ as an operator mapping polynomials of degree at most $n$ to $\mathbb{R}^{n+1}$ via $p\mapsto V^n(\bfx)\bfpi(p)$. Therefore, it makes sense to measure $p$ in the $L^2$ norm and $V^n(\bfx)\bfpi(p)$ in the 2-norm. By (\ref{eq:topology}), we can also consider $\bfpi(p)$ in the $M^n$-norm, which leads us to define the operator norm
\begin{equation}
\label{eq:Mto2}
\| A \|_{M^n\rightarrow 2} = \max_{\bfy\neq 0} \frac{ \| A\bfy \|_2 }{\| \bfy \|_{M^n}},
\end{equation}
and going in the opposite direction,
\begin{equation}
\label{eq:2toM}
\| A \|_{2\rightarrow M^n} = \max_{\bfy\neq 0} \frac{ \| A\bfy \|_{M^n}}{\| \bfy \|_2}.
\end{equation}
These two norms naturally combine to define a new condition number
\begin{equation}
\label{eq:kappaMto2}
\kappa_{M^n\rightarrow 2}(A) = \| A \|_{M^n\rightarrow 2} \left\| A^{-1} \right\|_{2\rightarrow M^n}.
\end{equation}

Since $M^n$ is symmetric and positive definite, it has a well-defined positive square root via the spectral decomposition.

\begin{lemma}
\label{lem:kapparelation}
\begin{equation}
\kappa_{M^n\rightarrow 2}\left(V^n(\bfx)\right) = \kappa_2 \left( V^n(\bfx)\left(M^n\right)^{-1/2}\right).
\end{equation}
\end{lemma}

\begin{proof}
Observe that if $\bfy\neq 0$, then
\begin{align*}
\frac{\left\| \left(V^n(\bfx)\right)^{-1} \bfy \right\|^2_{M^n}}{ \| \bfy \|^2_2} &= \frac{ \bfy^T\left(V^n(\bfx)\right)^{-T} M^n \left(V^n(\bfx)\right)^{-1}\bfy}{\| \bfy \|^2_2} \\
&= \frac{ \left( \left(M^n\right)^{1/2} \left(V^n(\bfx)\right)^{-1} \bfy\right)^T\left( \left(M^n\right)^{1/2} \left(V^n(\bfx)\right)^{-1} \bfy\right)}{\| \bfy \|^2_2} \\
&= \frac{ \left\| \left(M^n\right)^{1/2}\left( V^n(\bfx)\right)^{-1} \bfy\right\|^2_2}{ \| \bfy \|^2_2 },
\end{align*}
and so
\begin{equation}
\left\| \left(V^n(\bfx)\right)^{-1}\right\|_{2\rightarrow M^n} = \left\| \left(M^n\right)^{1/2}\left(V^n(\bfx)\right)^{-1} \right\|_2.
\end{equation}
Let $\bfz = \left(M^n\right)^{1/2} \bfy$. Then $\bfz\neq 0$ and
\begin{align*}
\frac{ \left\| V^n(\bfx)\bfy\right\|^2_2}{\left\| \bfy\right\|^2_{M^n}} &= \frac{ \bfy^T\left(V^n(\bfx)\right)^TV^n(\bfx)\bfy}{ \bfy^T M^n \bfy } \\
&= \frac{ \left( V^n(\bfx)\left(M^n\right)^{-1/2}\bfz\right)^T\left( V^n(\bfx)\left(M^n\right)^{-1/2}\bfz\right)}{\bfz^T\bfz} \\
&= \frac{ \left\| V^n(\bfx)\left(M^n\right)^{-1/2}\bfz\right\|^2_2}{ \| \bfz \|^2_2},
\end{align*}
and so
\begin{equation}
\left\| V^n(\bfx)\right\|_{M^n\rightarrow 2} = \left\| V^n(\bfx)\left(M^n\right)^{-1/2}\right\|_2.
\end{equation}
It follows that
\begin{align*}
\kappa_{M^n\rightarrow 2}\left(V^n(\bfx)\right) &= \left\| V^n(\bfx)\right\|_{M^n\rightarrow 2}\left\| \left(V^n(\bfx)\right)^{-1}\right\|_{2\rightarrow M^n} \\
&= \left\| V^n(\bfx)\left(M^n\right)^{-1/2}\right\|_2 \left\| \left(M^n\right)^{1/2}\left(V^n(\bfx)\right)^{-1} \right\|_2 \\
&=  \left\| V^n(\bfx)\left(M^n\right)^{-1/2}\right\|_2  \left\| \left(V^n(\bfx)\left(M^n\right)^{-1/2}\right)^{-1}\right\|_2 \\
&= \kappa_2\left( V^n(\bfx)\left(M^n\right)^{-1/2}\right).
\end{align*}
\end{proof}

To analyze $\kappa_{2}\left(V^n(\bfx)\left(M^n\right)^{-1/2}\right)$, we will need reference to the Legendre polynomials, mapped from their typical home on $[-1,1]$ to $[0,1]$. For $0\leq j\leq n$, let $L^j(x)$ denote the Legendre polynomial of degree $n$ over $[0,1]$, scaled so that $L^n(1)=1$ and
\begin{equation}
\label{eq:LegendreL2norm}
\| L^j \|^2_{L^2} = \frac{1}{2j+1}.
\end{equation}
The Legendre--Vandermonde matrix is the $(n+1)\times (n+1)$ matrix $\widehat{V}^n(\bfx)$ given by
\begin{equation}
\label{eq:LegendreVandermonde}
\widehat{V}^n_{ij}(\bfx) = L^j(\bfx_i).
\end{equation}
Given any polynomial $p$ of degree at most $n$, let $\bftheta(p)$ denote the vector of $n+1$ coefficients with respect to the Legendre basis. Define $T^n$ to be the $(n+1)\times (n+1)$ matrix satisfying $T^n\bftheta(p)= \bfpi(p)$ for all polynomials $p$ of degree at most $n$. In particular, we have the relationship
\begin{equation}
\label{eq:Vandermondeconversion}
\widehat{V}^n(\bfx) = V^n(\bfx)T^n.
\end{equation}

Let $\{\lambda^n_j\}_{j=0}^n$ be the set of eigenvalues of $M^n$. In \cite{allenkirby2020mass}, it was shown that $M^n$ admits the spectral decomposition
\begin{equation}
\label{eq:spectral}
M^n = Q^n\Lambda^n\left(Q^n\right)^T,
\end{equation}
where
\begin{equation}
Q^n = T^n\diag\left( \sqrt{(2j+1)\lambda^n_j}\right)_{j=0}^n
\end{equation}
is an orthogonal matrix and $\Lambda^n = \diag\left(\lambda^n_j\right)_{j=0}^n$ contains the eigenvalues. This decomposition combined with Lemma~\ref{lem:kapparelation} and the transformation given in (\ref{eq:Vandermondeconversion}) suggest a relationship between the $M^n\rightarrow 2$ condition number of $V^n(x)$ and the $2$-norm condition number of $\widehat{V}^n(\bfx)$ scaled by a diagonal matrix.

\begin{lemma}
\label{lem:kappa2Legendre}
\begin{equation}
\label{eq:kappa2Legendre}
\kappa_{M^n\rightarrow 2}\left(V^n(\bfx)\right) = \kappa_2\left( \widehat{V}^n(\bfx)\diag\left( \sqrt{2j+1}\right)_{j=0}^n\right).
\end{equation}
\end{lemma}

\begin{proof}
We have that
\begin{align*}
V^n(\bfx)\left(M^n\right)^{-1/2} &= V^n(\bfx)Q^n\left(\Lambda^n\right)^{-1/2}\left(Q^n\right)^T \\
&= V^n(\bfx)T^n\diag\left(\sqrt{(2j+1)\lambda^n_j}\right)_{j=0}^n\diag\left(\frac{1}{\sqrt{\lambda^n_j}}\right)_{j=0}^n\left(Q^n\right)^T \\
&= \widehat{V}^n(\bfx)\diag\left(\sqrt{2j+1}\right)_{j=0}^n\left(Q^n\right)^T.
\end{align*}
Since $Q^n$ is orthogonal, the result follows from Lemma~\ref{lem:kapparelation}.
\end{proof}

In \cite{gautschi1983condition}, Gautschi computed the condition number in the Frobenius norm $\| \cdot \|_F$ for Vandermonde matrices associated with families of orthonormal polynomials. Since the diagonal matrix in Lemma~\ref{lem:kappa2Legendre} contains the reciprocal of the $L^2$ norms of the Legendre polynomials, we can adapt his arguments to bound $\left(\widehat{V}^n(\bfx)\diag\left( \sqrt{2j+1}\right)_{j=0}^n\right)^{-1}$ in the Frobenius norm. In addition, the Legendre polynomials have the special property that $|L^j(x)|\leq 1$ for all $0\leq x\leq 1$, and so we can bound the Frobenius norm of $\widehat{V}^n(\bfx)\diag\left( \sqrt{2j+1}\right)_{j=0}^n$. 

For each integer $0\leq j\leq n$, let $\ell^{j,n}$ be the $j^{\text{th}}$ Lagrange polynomial with respect to $\bfx$; that is,
\begin{equation}
\ell^{j,n}(x) = \prod_{i\in\{0,\dots,n\}\setminus\{j\}} \frac{x-\bfx_i}{\bfx_j-\bfx_i}.
\end{equation}
Define $\bfw^n\in\mathbb{R}^{n+1}$ by
\begin{equation}
\bfw^n_j = \| \ell^{j,n} \|_{L^2}.
\end{equation}

\begin{theorem}
\label{thm:kappaMto2}
\begin{equation}
\kappa_{M^n\rightarrow 2}(V^n(\bfx)) \leq (n+1)^{3/2}\| \bfw^n \|_2.
\end{equation}
\end{theorem}

\begin{proof}
Since $|L^j(x)|\leq 1$ for all $0\leq x\leq 1$, we have that
\begin{align*}
\left\| \widehat{V}^n(\bfx)\diag\left(\sqrt{2j+1}\right)_{j=0}^n\right\|_F &= \left( \sum_{i,j=0}^n (2j+1)\left(L^j(\bfx_i)\right)^2\right)^{1/2} \\
&\leq \left( \sum_{i,j=0}^n (2j+1) \right)^{1/2} \\
&= (n+1)^{3/2}.
\end{align*}
Since 
\begin{equation}
\ell^{j,n}(\bfx_i) = \begin{cases} 1, & \text{if}\ i=j; \\ 0, & \text{if}\ i\neq j; \end{cases}
\end{equation}
we have that $\left(\widehat{V}^n(\bfx)\right)^{-1}_{ij} = \bftheta\left(\ell^{j,n}\right)_i$. Therefore, by (\ref{eq:LegendreL2norm}), we have that
\begin{align*}
\| \bfw^n \|_2 &= \left( \sum_{j=0}^n \left(\bfw^n_j\right)^2\right)^{1/2} \\
&= \left(\sum_{j=0}^n \int_0^1 \left(\ell^{j,n}(x)\right)^2dx\right)^{1/2} \\
&= \left(\sum_{j=0}^n \int_0^1 \sum_{i,k=0}^n \bftheta(\ell^{j,n})_i\bftheta(\ell^{j,n})_kL^i(x)L^k(x)dx\right)^{1/2} \\ 
&= \left(\sum_{j=0}^n \sum_{i,k=0}^n \bftheta(\ell^{j,n})_i\bftheta(\ell^{j,n})_k \int_0^1 L^i(x)L^k(x)dx\right)^{1/2} \\ 
&= \left(\sum_{i,j=0}^n \bftheta(\ell^{j,n})^2_i \int_0^1 (L^i(x))^2dx\right)^{1/2} \\ 
&= \left(\sum_{i,j=0}^n \frac{\bftheta(\ell^{j,n})^2_i}{2i+1}\right)^{1/2} \\
&= \left\| \diag\left( \frac{1}{\sqrt{2i+1}}\right)_{i=0}^n \left(\widehat{V}^n(\bfx)\right)^{-1}\right\|_F \\
&= \left\| \left(\widehat{V}^n(\bfx)\diag\left(\sqrt{2j+1}\right)_{j=0}^n\right)^{-1}\right\|_F.
\end{align*}
Since the condition number in the 2 norm is bounded above by the condition number in the Frobenius norm, the result follows from Lemma~\ref{lem:kappa2Legendre}.
\end{proof}

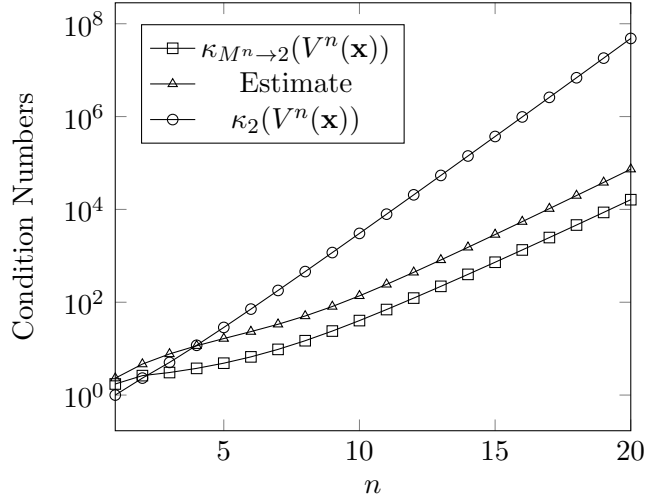
\begin{figure}
\begin{center}
\begin{tikzpicture}
\begin{semilogyaxis}[legend style={at={(0.05,0.95)}, anchor=north west}, xlabel=$n$, ylabel=Condition Numbers, xmin=1, xmax=20, ymin=0]
\addplot [mark=square] table [x=n, y=k2LD, col sep=comma] {conditioning.dat};
\addplot [mark=triangle] table [x=n, y=ub, col sep=comma] {conditioning.dat};
\addplot [mark=o] table [x=n, y=k2V, col sep=comma] {conditioning.dat};
\legend{$\kappa_{M^n\rightarrow 2}(V^n(\bfx))$, Estimate, $\kappa_2(V^n(\bfx))$}
\end{semilogyaxis}
\end{tikzpicture}
\end{center}
\caption{Comparsion of the condition number of $V^n$ in the $M^n\rightarrow 2$ norm and the $2$ norm for $1\leq n\leq 20$, where $V^n$ is the matrix described in Subsection~\ref{ssec:equispaced}. We also include the estimate given in Theorem~\ref{thm:kappaMto2}. We use Lemma~\ref{lem:kappa2Legendre} to compute $\kappa_{M^n\rightarrow 2}(V^n(\bfx))$.}
\label{fig:kappa}
\end{figure}

\section{Higher dimensions}
\label{sec:dimensions}

Now, in the case of equispaced nodes on the $d$-simplex, we can use block-recursive structure as discussed in \cite{kirby2012fast} to give low-complexity algorithms for the inversion of simplicial Bernstein--Vandermonde matrices. This gives, at least for the equispaced lattice, an alternate algorithm for solving simplicial Bernstein interpolation problems to the one worked out in \cite{ainsworth2011bernstein}.

\subsection{Notation and Preliminaries}
\label{ssec:notation}

For an integer $d\geq 1$, let $S_d$ be a nondegenerate simplex in $\mathbb{R}^d$. Let $\{\bfv_i\}_{i=0}^d\subset \mathbb{R}^d$ be the vertices of $S_d$, and let $\{\bfb_i\}_{i=0}^d$ denote the barycentric coordinates of $S_d$. Each $\bfb_i$ is an affine map from $\mathbb{R}^d$ to $\mathbb{R}$ such that
\begin{equation}
\bfb_i(\bfv_j) = \begin{cases} 1, & \text{if}\ i=j; \\ 0, & \text{if}\ i\neq j; \end{cases}
\end{equation}
for each vertex $\bfv_j$. Each $\bfb_i$ is nonnegative on $S_d$, and
\begin{equation}
\sum_{i=0}^d \bfb_i = 1.
\end{equation}
A multiindex $\bfbeta$ of length $d+1$ is a $(d+1)$-tuple of nonnegative integers, written
\begin{equation}
\bfbeta = (\bfbeta_0,\dots,\bfbeta_d).
\end{equation}
The order of $\bfbeta$, denoted $|\bfbeta|$, is given by
\begin{equation}
|\bfbeta| = \sum_{i=0}^d \bfbeta_i.
\end{equation}
The factorial $\bfbeta!$ of a multiindex $\bfbeta$ is defined by
\begin{equation}
\bfbeta! = \prod_{i=0}^d \bfbeta_i.
\end{equation}
For a multiindex $\bfbeta$ of length $d+1$, denote by $\bfbeta'$ the multiindex of length $d$ given by
\begin{equation}
\bfbeta' = (\bfbeta_1, \dots, \bfbeta_d).
\end{equation}
Given a nonnegative integer $b$ and a multiindex $\bfbeta' = (\bfbeta_1,\dots, \bfbeta_d)$ of length $d$, define a new multiindex $b\vdash\bfbeta$ of length $d+1$ by
\begin{equation}
b\vdash\bfbeta' = (b,\bfbeta_1,\dots,\bfbeta_d).
\end{equation}
In particular,
\begin{equation}
\bfbeta = \bfbeta_0\vdash\bfbeta'.
\end{equation}
Multiindices have a natural partial ordering given by
\begin{equation}
\label{eq:ordering}
\bfbeta \leq \widetilde{\bfbeta}\qquad \text{if and only if}\qquad \bfbeta_i\leq \widetilde{\bfbeta}_i\quad \text{for all}\quad 0\leq i\leq d.
\end{equation}
The Bernstein polynomials of degree $n$ on the $d$-simplex $S_d$ are defined by
\begin{equation}
B^n_{\bfbeta} = \frac{n!}{\bfbeta!} \prod_{i=0}^d \bfb_i^{\bfbeta_i}.
\end{equation}
The complete set of Bernstein polynomials $\{B^n_{\bfbeta}\}_{|\bfbeta|=n}$ form a basis for polynomials in $d$ variables of complete degree at most $n$.

If $n_0\leq n$, then any polynomial expressed in the basis $\{B^{n_0}_{\bfbeta}\}_{|\bfbeta|=n_0}$ can also be expressed in the basis $\{B^n_{\bfbeta}\}_{|\bfbeta|=n}$. We denote by $E^{d,n_0,n}$ the $\binom{n+d}{d}\times\binom{n_0+d}{d}$ matrix that maps the coefficients of the degree $n_0$ representation to the coefficients of the degree $n$ representation. The matrix $E^{d,n_0,n}$ is sparse and can be applied matrix-free \cite{kirby2017fast}, if desired.

For a nonnegative integer $m$ and a set of distinct nodes $\{\bfx^m_{\bfalpha}\}_{|\bfalpha|=m}\subset S_d$, define the Bernstein--Vandermonde matrix $V^{d,m,n}$ to be the $\binom{m+d}{d}\times\binom{n+d}{d}$ matrix given by
\begin{equation}
V^{d,m,n}_{\bfalpha\bfbeta} = B^n_{\bfbeta}(\bfx^m_{\bfalpha})
\end{equation}

\begin{theorem}
\label{thm:elev}
Let $d,m,n_0$, and $n$ be nonnegative integers with $d\geq 1$ and $n_0\leq n$, and let $\{\bfx^m_{\bfalpha}\}_{|\bfalpha|=m}$ be distinct nodes contained in the $d$-simplex $S_d$. Then
\begin{equation}
\label{eq:elev}
V^{d,m,n}E^{d,n_0,n}=V^{d,m,n_0}.
\end{equation}
\end{theorem}

\begin{proof}
Given a polynomial $p\in\Span\{B^{n_0}_{\bfbeta}\}_{|\bfbeta|=n_0}$, the matrix $V^{d,m,n_0}$ evaluates $p$ at each of the nodes $\bfx^m_{\bfalpha}$. On the other hand, the matrix $V^{d,m,n}E^{d,n_0,n}$ evaluates the degree $n$ representation of $p$ at each of the nodes $\bfx^m_{\bfalpha}$. Since we are evaluating at the same nodes and the polynomial has not been modified, both evaluations must give the same result.
\end{proof}

The partial ordering (\ref{eq:ordering}) implies a natural way to order the entries of $V^{d,m,n}$ and also imposes a block structure on $V^{d,m,n}$ by dividing the matrix into sections where $\bfalpha_0$ and $\bfbeta_0$ are constant. For integers $0\leq \bfalpha_0\leq m$ and $0\leq \bfbeta_0\leq n$, let $V^{d,m,n}_{\bfalpha_0\bfbeta_0}$ denote the $\binom{m-\bfalpha_0+d}{d}\times\binom{n-\bfbeta_0+d}{d}$ submatrix of $V^{d,m,n}$ whose entries satisfy
\begin{equation}
\left(V^{d,m,n}_{\bfalpha_0\bfbeta_0}\right)_{\bfalpha'\bfbeta'} = V^{d,m,n}_{\left(\bfalpha_0\vdash\bfalpha'\right)\left(\bfbeta_0\vdash\bfbeta'\right)}.
\end{equation}

\subsection{Degree Reduction}
\label{ssec:reduction}

We now consider the case of equispaced nodes on the $d$-simplex $S^d$. For this case, the entries of the Bernstein--Vandermonde matrix are given by
\begin{equation}
\label{eq:entries}
V^{d,m,n}_{\bfalpha\bfbeta} = \frac{n!}{\bfbeta!}\frac{1}{m^n} \prod_{i=0}^d \bfalpha_i^{\bfbeta_i}.
\end{equation}
This representation of the entries allows us to represent the submatrix $V^{d,m,n}_{\bfalpha_0\bfbeta_0}$ in terms of a lower-dimensional Bernstein--Vandermonde matrix.

\begin{theorem}
\label{thm:reduction}
Let $d>1$ and $m,n\geq 0$ be integers, and let $\bfalpha_0,\bfbeta_0$ be integers with $0\leq\bfalpha_0\leq m$ and $0\leq\bfbeta_0\leq n$. Then
\begin{equation}
\label{eq:reduction}
V^{d,m,n}_{\bfalpha_0\bfbeta_0} = m_{\bfalpha_0\bfbeta_0}V^{d-1,m-\bfalpha_0,n-\beta_0},
\end{equation}
where
\begin{equation}
m_{\bfalpha_0\bfbeta_0} = \binom{n}{\bfbeta_0}(\bfalpha_0/m)^{\bfbeta_0}(1-\bfalpha_0/m)^{n-\bfbeta_0}.
\end{equation}
\end{theorem}

\begin{proof}
If $\bfalpha_0=m$ and $\bfbeta_0\neq n$, then both sides equal zero and we are done. Otherwise, by (\ref{eq:entries}), we have that
\begin{equation}
V^{d-1,m-\bfalpha_0,n-\bfbeta_0}_{\bfalpha'\bfbeta'} = \frac{(n-\bfbeta_0)!}{\bfbeta'!}\frac{1}{(m-\bfalpha_0)^{n-\bfbeta_0}}\prod_{i=1}^d \bfalpha_i^{\bfbeta_i}.
\end{equation}
On the other hand,
\begin{align*}
\left(V^{d,m,n}_{\bfalpha_0\bfbeta_0}\right)_{\bfalpha'\bfbeta'} &= V^{d,m,n}_{\bfalpha\bfbeta} \\
&= \frac{n!}{\bfbeta!}\frac{1}{m^n} \prod_{i=0}^d \bfalpha_i^{\bfbeta_i} \\
&= \frac{n!}{\bfbeta_0!(n-\bfbeta_0)!}\frac{(n-\bfbeta_0)!}{\bfbeta'!} \frac{\bfalpha_0^{\bfbeta_0}}{m^{\bfbeta_0}} \frac{(m-\bfalpha_0)^{n-\bfbeta_0}}{m^{n-\bfbeta_0}} \frac{1}{(m-\bfalpha_0)^{n-\bfbeta_0}}\prod_{i=1}^d \bfalpha_i^{\bfbeta_i},
\end{align*}
where we have separated terms containing $\bfalpha_0$ and $\bfbeta_0$ and multiplied and divided  by $(n-\bfbeta_0)!$ and $(m-\bfalpha_0)^{n-\bfbeta_0}$. Since
\begin{equation}
\frac{n!}{\bfbeta_0!(n-\bfbeta_0)!}\frac{\bfalpha_0^{\bfbeta_0}}{m^{\bfbeta_0}}\frac{(m-\alpha_0)^{n-\bfbeta_0}}{m^{n-\bfbeta_0}} = \binom{n}{\bfbeta_0}(\bfalpha_0/m)^{\bfbeta_0}(1-\bfalpha_0/m)^{n-\bfbeta_0},
\end{equation}
we have the desired result.
\end{proof}

The fact that the blocks are multiples of lower-dimensional matrices is analogous to what has been observed for other matrices related to Bernstein polynomials. For example, the Bernstein--Vandermonde matrix associated with the zeroes of Legendre polynomials has a similar structure \cite{kirby2012fast}, and so does the Bernstein mass matrix \cite{kirby2017fast}. 

We recognize that $m_{\bfalpha_0\bfbeta_0}$ describes a univariate Bernstein polynomial of degree $n$ being evaluated at equispaced points. Therefore, if $V^n$ is the one-dimensional Bernstein--Vandermonde matrix associated with equispaced nodes as described in Subsection~\ref{ssec:equispaced}, then we have that $V^{d,n,n}$ admits the block structure
\begin{equation}
\label{eq:block}
V^{d,n,n} = 
\begin{pmatrix}
V^n_{00}V^{d-1,n,n} & V^n_{01}V^{d-1,n,n-1} & \cdots & V^n_{0n}V^{d-1,n,0} \\
V^n_{10}V^{d-1,n-1,n} & V^n_{11}V^{d-1,n-1,n-1} & \cdots & V^n_{1n}V^{d-1,n-1,0} \\
\vdots & \vdots & \ddots & \vdots \\
V^n_{n0}V^{d-1,0,n} & V^n_{n1}V^{d-1,0,n-1} & \cdots & V^n_{nn}V^{d-1,0,0}
\end{pmatrix}.
\end{equation}
We can use this block structure and Theorem~\ref{thm:elev} to perform block Gaussian elimination on $V^{d,n,n}$. This method was used in \cite{kirby2017fast} to obtain the block $LU$ decomposition of the Bernstein mass matrix. In the same way, we have the following:

\begin{theorem}
\label{thm:LU}
Let $V^n$ be the one-dimensional Bernstein--Vandermonde matrix associated with equispaced nodes as described in Subsection~\ref{ssec:equispaced}. Suppose $V^n=L^nU^n$ is the $LU$ decomposition of $V^n$. Then
\begin{equation}
V^{d,n,n}=L^{d,n}U^{d,n},
\end{equation}
where $L^{d,n}$ is the block lower triangular matrix with blocks given by
\begin{equation}
L^{d,n}_{\bfalpha_0\bfbeta_0} = L^n_{\bfalpha_0\bfbeta_0}V^{d-1,n-\bfalpha_0,n-\bfbeta_0}
\end{equation}
and $U^{d,n}$ is the block upper triangular matrix with blocks given by
\begin{equation}
U^{d,n}_{\bfalpha_0\bfbeta_0} = U^n_{\bfalpha_0\bfbeta_0}E^{d-1,n-\bfbeta_0,n-\bfalpha_0}.
\end{equation}
\end{theorem}

\begin{proof}
By Theorem~\ref{thm:elev}, if $0\leq \bfalpha_0,\bfbeta_0\leq n$ and $\bm{\gamma}_0\leq \bfbeta_0$, then
\begin{align*}
L^{d,n}_{\bfalpha_0\bm{\gamma}_0}U^{d,n}_{\bm{\gamma}_0\bfbeta_0} &=  L^n_{\bfalpha_0\bm{\gamma}_0}U^n_{\bm{\gamma}_0\bfbeta_0}V^{d-1,n-\bfalpha_0,n-\bm{\gamma}_0}E^{d-1,n-\bfbeta_0,n-\bm{\gamma}_0} \\
&= L^n_{\bfalpha_0\bm{\gamma}_0}U^n_{\bm{\gamma}_0\bfbeta_0} V^{d-1,n-\bfalpha_0,n-\bfbeta_0}.
\end{align*}
Therefore,
\begin{align*}
\left(L^{d,n}U^{d,n}\right)_{\bfalpha_0\bfbeta_0} &= \sum_{\bm{\gamma}_0=0}^{\bfbeta_0} L^{d,n}_{\bfalpha_0\bm{\gamma}_0}U^{d,n}_{\bm{\gamma}_0\bfbeta_0} \\
&= \left( \sum_{\bm{\gamma}_0=0}^{\bfbeta_0} L^n_{\bfalpha_0\bm{\gamma}_0}U^n_{\bm{\gamma}_0\bfbeta_0}\right) V^{d-1,n-\bfalpha_0,n-\bfbeta_0} \\
&= V^n_{\bfalpha_0\bfbeta_0}V^{d-1,n-\bfalpha_0,n-\bfbeta_0} \\
&= V^{d,n,n}_{\bfalpha_0\bfbeta_0}.
\end{align*}
\end{proof}

\section{Numerical results}

Now, we consider the accuracy of the methods described above on several problems. For all of the problems, we chose random solution vectors, computed the right-hand side by matrix multiplication, and then attempted to recover the solution. The numerical results are run in double precision arithmetic on a 2014 Macbook Air running macOS 10.13 and using Python 2.7.13. Cholesky factorization and FFTs are performed using the \texttt{numpy} (v1.12.0) function calls.  Also, because our code is a mix of pure Python and low-level compiled libraries, timings are not terribly informative.  Consequently, we focus on assessing the stability and accuracy of our methods.  If future work leads to more stable fast algorithms, greater care will be afforded to tuning our implementations for performance.

In Figure~\ref{fig:equispacedinfo}, we consider the case of equispaced nodes. When considering the Euclidean norm of the error (Figure~\ref{subfig:equispaced2err}), the $LU$ factorization and the Newton algorithm have the best perfomance, followed by the DFT-based application and multiplying by the inverse. We observe that the $LU$ and Newton methods have comparable performance, as do the DFT-based algorithm and multiplying by the inverse. This is in contrast to what was observed in \cite{allenkirby2020mass}, where a similar DFT-based algorithm quickly became unstable. The same behavior can be observed when comparing the Euclidean norm of the residual (Figure~\ref{subfig:equispaced2res}), although some separation does occur between the Newton method and the $LU$ factorization. Since the solution vector can be viewed as the Bernstein coefficients of the interpolation polynomial, we also measure the $L^2$ difference between the exact and computed solutions (Figure~\ref{subfig:equispacedMerr}); equivalently, we measure the relative $M^n$ error, where $M^n$ is the Bernstein mass matrix given in (\ref{eq:mass}). All four solution methods have very similar behavior in the $M^n$ norm as they do when considering the Euclidean norm of the residual.

To ensure that the behavior of the solution methods does not depend on the choice of nodes, we also considered the case where the nodes $\bfx_j$ are randomly selected from $[j/(n+1),(j+1)/(n+1))$ for each $0\leq j\leq n$ (Figure~\ref{fig:randominfo}); however, there was no significant difference in the quantities measured between this case and the equispaced case.

We also used the block $LU$ decomposition given in Theorem~\ref{thm:LU} combined with the one-dimensional $LU$ algorithm to solve the interpolation problem for equispaced nodes on the $d$-simplex for $d=2$ and $d=3$ (Figure~\ref{fig:blockinfo}). Due to the recursive nature of the algorithm, the quantities measured are very similar to the ones observed for equispaced nodes.

\begin{figure}
\centering
\begin{subfigure}{0.475\linewidth}
\centering
\begin{tikzpicture}[scale=0.7]
\begin{semilogyaxis}[legend style={at={(0.05,0.95)}, anchor=north west}, xlabel=$n$, ylabel=$\| \bfc-\widehat{\bfc} \|_2 / \| \bfc \|_2$, xmin=1, xmax=20, ymin=1e-18]
\addplot [mark=square] table [x=n, y=BezoutL2err, col sep=comma] {equispaced_info.dat};
\addplot [mark=triangle] table [x=n, y=DFTL2err, col sep=comma] {equispaced_info.dat};
\addplot [mark=diamond] table [x=n, y=LUL2err, col sep=comma] {equispaced_info.dat};
\addplot[mark=o] table [x=n, y=NewtonL2err, col sep=comma]{equispaced_info.dat};
\legend{B\'{e}zout, DFT, $LU$, Newton}
\end{semilogyaxis}
\end{tikzpicture}
\caption{Error in the 2-norm.}
\label{subfig:equispaced2err}
\end{subfigure}
\hfill
\begin{subfigure}{0.475\linewidth}
\centering
\begin{tikzpicture}[scale=0.7]
\begin{semilogyaxis}[legend style={at={(0.05,0.95)}, anchor=north west}, xlabel=$n$, ylabel=$\| \bfc-\widehat{\bfc} \|_{M^n} / \| \bfc \|_{M^n}$, xmin=1, xmax=20, ymin=1e-18]
\addplot [mark=square] table [x=n, y=BezoutMerr, col sep=comma] {equispaced_info.dat};
\addplot [mark=triangle] table [x=n, y=DFTMerr, col sep=comma] {equispaced_info.dat};
\addplot [mark=diamond] table [x=n, y=LUMerr, col sep=comma] {equispaced_info.dat};
\addplot[mark=o] table [x=n, y=NewtonMerr, col sep=comma]{equispaced_info.dat};
\legend{B\'{e}zout, DFT, $LU$, Newton}
\end{semilogyaxis}
\end{tikzpicture}
\caption{Error in the $M^n$ norm.}
\label{subfig:equispacedMerr}
\end{subfigure}
\vskip\baselineskip
\begin{subfigure}{\linewidth}
\centering
\begin{tikzpicture}[scale=0.7]
\begin{semilogyaxis}[legend style={at={(0.05,0.95)}, anchor=north west}, xlabel=$n$, ylabel=$\| V^n\widehat{\bfc}-\bfb \|_2$, xmin=1, xmax=20, ymin=1e-18]
\addplot [mark=square] table [x=n, y=Bezoutres, col sep=comma] {equispaced_info.dat};
\addplot [mark=triangle] table [x=n, y=DFTres, col sep=comma] {equispaced_info.dat};
\addplot [mark=diamond] table [x=n, y=LUres, col sep=comma] {equispaced_info.dat};
\addplot[mark=o] table [x=n, y=Newtonres, col sep=comma]{equispaced_info.dat};
\legend{B\'{e}zout, DFT, $LU$, Newton}
\end{semilogyaxis}
\end{tikzpicture}
\caption{Residual in the 2-norm.}
\label{subfig:equispaced2res}
\end{subfigure}
\caption{Error/residual in using the methods described in Section~\ref{sec:apply} to solve $V^n\bfc=\bfb$ for $1\leq n\leq 20$, where $V^n$ is the matrix described in Subsection~\ref{ssec:equispaced} and $\bfb$ is a random vector in $[-1,1]^{n+1}$. B\'{e}zout refers to Corollary~\ref{cor:bezout}, DFT refers to Corollary~\ref{cor:equispacedinverse}, $LU$ refers to $LU$ decomposition of $V^n$, and Newton refers to the Ainsworth--Sanchez algorithm. We use $\widehat{\bfc}$ to denote the computed solution.}
\label{fig:equispacedinfo}
\end{figure}
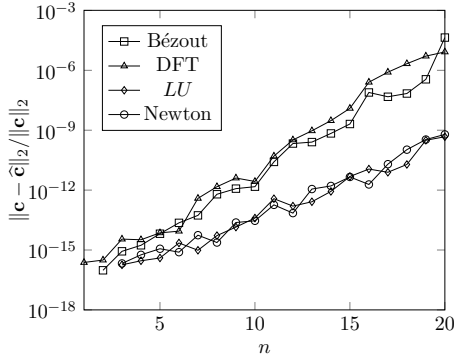
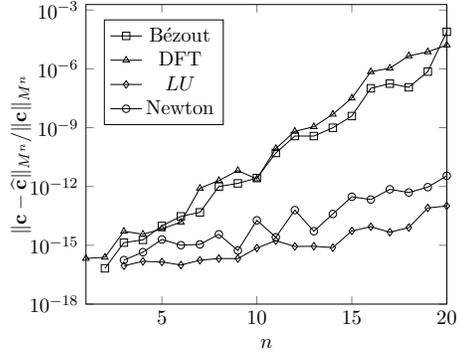
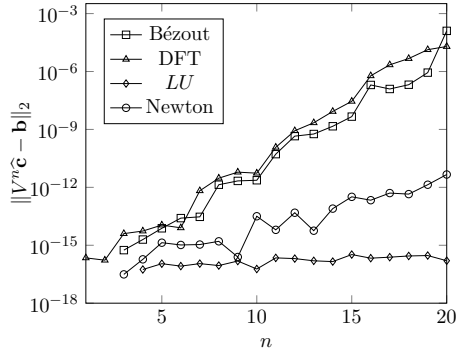

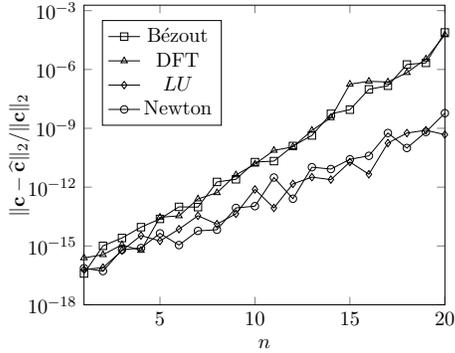
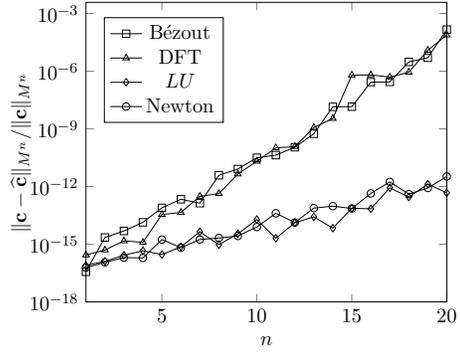
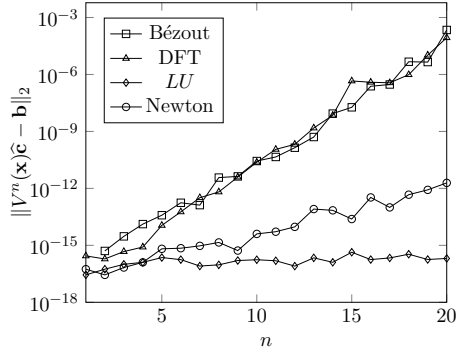
\begin{figure}
\centering
\begin{subfigure}{0.475\linewidth}
\centering
\begin{tikzpicture}[scale=0.7]
\begin{semilogyaxis}[legend style={at={(0.05,0.95)}, anchor=north west}, xlabel=$n$, ylabel=$\| \bfc-\widehat{\bfc} \|_2 / \| \bfc \|_2$, xmin=1, xmax=20, ymin=1e-18]
\addplot [mark=square] table [x=n, y=BezoutL2err, col sep=comma] {random_info.dat};
\addplot [mark=triangle] table [x=n, y=DFTL2err, col sep=comma] {random_info.dat};
\addplot [mark=diamond] table [x=n, y=LUL2err, col sep=comma] {random_info.dat};
\addplot[mark=o] table [x=n, y=NewtonL2err, col sep=comma]{random_info.dat};
\legend{B\'{e}zout, DFT, $LU$, Newton}
\end{semilogyaxis}
\end{tikzpicture}
\caption{Error in the 2-norm.}
\label{subfig:random2err}
\end{subfigure}
\hfill
\begin{subfigure}{0.475\linewidth}
\centering
\begin{tikzpicture}[scale=0.7]
\begin{semilogyaxis}[legend style={at={(0.05,0.95)}, anchor=north west}, xlabel=$n$, ylabel=$\| \bfc-\widehat{\bfc} \|_{M^n} / \| \bfc \|_{M^n}$, xmin=1, xmax=20, ymin=1e-18]
\addplot [mark=square] table [x=n, y=BezoutMerr, col sep=comma] {random_info.dat};
\addplot [mark=triangle] table [x=n, y=DFTMerr, col sep=comma] {random_info.dat};
\addplot [mark=diamond] table [x=n, y=LUMerr, col sep=comma] {random_info.dat};
\addplot[mark=o] table [x=n, y=NewtonMerr, col sep=comma]{random_info.dat};
\legend{B\'{e}zout, DFT, $LU$, Newton}
\end{semilogyaxis}
\end{tikzpicture}
\caption{Error in the $M^n$ norm.}
\label{subfig:randomMerr}
\end{subfigure}
\vskip\baselineskip
\begin{subfigure}{\linewidth}
\centering
\begin{tikzpicture}[scale=0.7]
\begin{semilogyaxis}[legend style={at={(0.05,0.95)}, anchor=north west}, xlabel=$n$, ylabel=$\| V^n(\bfx)\widehat{\bfc}-\bfb \|_2$, xmin=1, xmax=20, ymin=1e-18]
\addplot [mark=square] table [x=n, y=Bezoutres, col sep=comma] {random_info.dat};
\addplot [mark=triangle] table [x=n, y=DFTres, col sep=comma] {random_info.dat};
\addplot [mark=diamond] table [x=n, y=LUres, col sep=comma] {random_info.dat};
\addplot[mark=o] table [x=n, y=Newtonres, col sep=comma]{random_info.dat};
\legend{B\'{e}zout, DFT, $LU$, Newton}
\end{semilogyaxis}
\end{tikzpicture}
\caption{Residual in the 2-norm.}
\label{subfig:random2res}
\end{subfigure}
\caption{Error/residual in using the methods described in Section~\ref{sec:apply} to solve $V^n(\bfx)\bfc=\bfb$ for $1\leq n\leq 20$, where $V^n(\bfx)$ is the Bernstein--Vandermonde matrix associated to $\bfx$, the nodes $\bfx_j$ are randomly selected from $[j/(n+1),(j+1)/(n+1))$ for $0\leq j\leq n$, and $\bfb$ is a random vector in $[-1,1]^{n+1}$. B\'{e}zout refers to Corollary~\ref{cor:bezout}, DFT refers to Theorem~\ref{thm:factoredinverse}, $LU$ refers to $LU$ decomposition of $V^n$, and Newton refers to the Ainsworth--Sanchez algorithm. We use $\widehat{\bfc}$ to denote the computed solution.}
\label{fig:randominfo}
\end{figure}

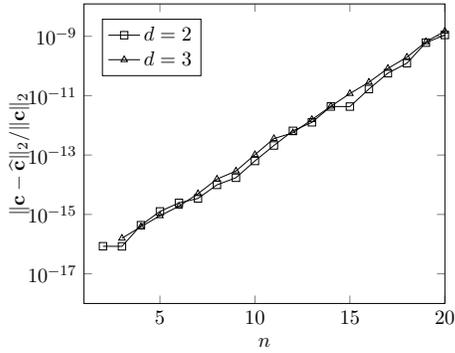
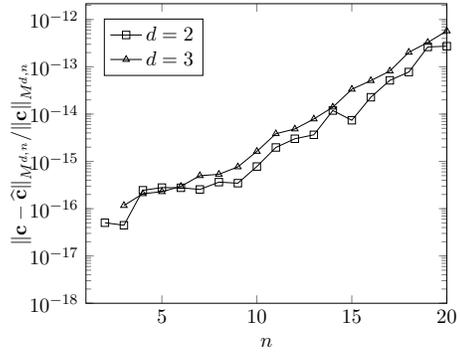
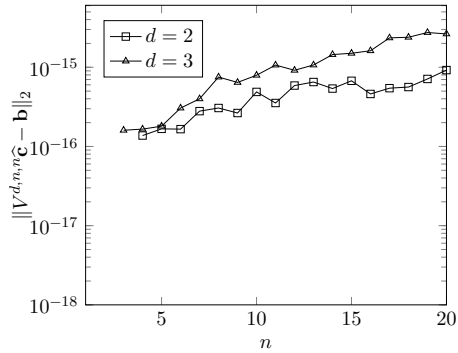
\begin{figure}
\centering
\begin{subfigure}{0.475\linewidth}
\centering
\begin{tikzpicture}[scale=0.7]
\begin{semilogyaxis}[legend style={at={(0.05,0.95)}, anchor=north west}, xlabel=$n$, ylabel=$\| \bfc-\widehat{\bfc} \|_2 / \| \bfc \|_2$, xmin=1, xmax=20, ymin=1e-18]
\addplot [mark=square] table [x=n, y=2dL2err, col sep=comma] {blockLU_info.dat};
\addplot [mark=triangle] table [x=n, y=3dL2err, col sep=comma] {blockLU_info.dat};
\legend{$d=2$, $d=3$}
\end{semilogyaxis}
\end{tikzpicture}
\caption{Error in the 2-norm.}
\label{subfig:block2err}
\end{subfigure}
\hfill
\begin{subfigure}{0.475\linewidth}
\centering
\begin{tikzpicture}[scale=0.7]
\begin{semilogyaxis}[legend style={at={(0.05,0.95)}, anchor=north west}, xlabel=$n$, ylabel=$\| \bfc-\widehat{\bfc} \|_{M^{d,n}} / \| \bfc \|_{M^{d,n}}$, xmin=1, xmax=20, ymin=1e-18]
\addplot [mark=square] table [x=n, y=2dMerr, col sep=comma] {blockLU_info.dat};
\addplot [mark=triangle] table [x=n, y=3dMerr, col sep=comma] {blockLU_info.dat};
\legend{$d=2$, $d=3$}
\end{semilogyaxis}
\end{tikzpicture}
\caption{Error in the $M^{d,n}$ norm.}
\label{subfig:blockMerr}
\end{subfigure}
\vskip\baselineskip
\begin{subfigure}{\linewidth}
\centering
\begin{tikzpicture}[scale=0.7]
\begin{semilogyaxis}[legend style={at={(0.05,0.95)}, anchor=north west}, xlabel=$n$, ylabel=$\| V^{d,n,n}\widehat{\bfc}-\bfb \|_2$, xmin=1, xmax=20, ymin=1e-18]
\addplot [mark=square] table [x=n, y=2dres, col sep=comma] {blockLU_info.dat};
\addplot [mark=triangle] table [x=n, y=3dres, col sep=comma] {blockLU_info.dat};
\legend{$d=2$, $d=3$}
\end{semilogyaxis}
\end{tikzpicture}
\caption{Residual in the 2-norm.}
\label{subfig:block2res}
\end{subfigure}
\caption{Error/residual in using the block $LU$ decomposition given in Theorem~\ref{thm:LU} to solve $V^{d,n,n}\bfc=\bfb$ for $1\leq n\leq 20$ and $d=2,3$, where $V^{d,n,n}$ is the Bernstein--Vandermonde matrix given in (\ref{eq:entries}) and $\bfb$ is a random vector in $[-1,1]^{\binom{n+d}{d}}$. We use $\widehat{\bfc}$ to denote the computed solution.}
\label{fig:blockinfo}
\end{figure}

\section{Conclusion}

We have studied several algorithms for the inversion of the univariate Bernstein--Vandermonde matrix. These algorithms, while less stable than algorithms discovered previously, provide insight into the structure of the Bernstein--Vandermonde matrix and are remarkably similar to algorithms derived for the Bernstein mass matrix. In addition, we have used a block $LU$ decomposition of the Bernstein--Vandermonde matrix corresponding to equispaced nodes on the $d$-simplex to give a recursive, block-structured algorithm with comparable accuracy to the one-dimensional algorithm. Moreover, we have given a new perspective on the conditioning of the Bernstein--Vandermonde matrix, indicating that the interpolation problem is better-conditioned with respect to the $L^2$ norm than the Euclidean norm. In the future, we hope to expand this perspective to other polynomial problems and continue the development of fast and accurate methods for problems involving Bernstein polynomials.

\clearpage

\bibliographystyle{plain}
\bibliography{references}
\end{document}